\documentclass[12 pt]{amsart}
\usepackage{amsmath,times,epsfig,amssymb,amsbsy,amscd,amsfonts,amstext,graphicx,color,bm,amsthm}
\usepackage[all]{xy}
\usepackage{graphicx}

\usepackage[urlcolor=blue]{hyperref}
\usepackage{tikz,subfigure}

\newcommand{\C}{\mathbb{C}}

\newcommand{\R}{\mathbb{R}}

\renewcommand{\to}{\longrightarrow}

\newtheorem{Theorem}{Theorem}[section]
\newtheorem{Definition}[Theorem]{Definition}

\newtheorem{Lemma}[Theorem]{Lemma}
\newtheorem{Proposition}[Theorem]{Proposition}
\newtheorem{Corollary}[Theorem]{Corollary}
\newtheorem{Remark}[Theorem]{Remark}
\newtheorem{Example}[Theorem]{Example}

\DeclareMathOperator{\Der}{Der}

\DeclareMathOperator{\depth}{depth}
\DeclareMathOperator{\hdim}{hdim}
\DeclareMathOperator{\pdeg}{pdeg}

\DeclareMathOperator{\codim}{codim}

\def \X(#1){\{x_1,\dots, x_{#1}\}}

\def \degrevlex{{\rm degrevlex}}

\DeclareMathOperator{\LM}{LM}
\DeclareMathOperator{\LC}{LC}
\DeclareMathOperator{\rk}{rk}

\newcommand{\A}{\mathcal{A}}

\DeclareMathOperator{\charact}{char}

\begin{document}

\title{Free hyperplane arrangements over arbitrary fields}

\begin{abstract}
In this paper, we study the class of free hyperplane arrangements. Specifically, we investigate the relations between freeness over a field of finite characteristic
and freeness over $\mathbb{Q}$.
\end{abstract}

\author{Elisa Palezzato}
\address{Elisa Palezzato, Department of Mathematics, Hokkaido University, Kita 10, Nishi 8, Kita-Ku, Sapporo 060-0810, Japan.}
\email{palezzato@math.sci.hokudai.ac.jp}
\author{Michele Torielli}
\address{Michele Torielli, Department of Mathematics, Hokkaido University, Kita 10, Nishi 8, Kita-Ku, Sapporo 060-0810, Japan.}
\email{torielli@math.sci.hokudai.ac.jp}


\date{\today}
\maketitle

\tableofcontents

\section{Introduction}
Let $V$ be a vector space of dimension $l$ over a field $K$. Fix a system of coordinate $(x_1,\dots, x_l)$ of $V^\ast$. 
We denote by $S = S(V^\ast) = K[x_1,\dots, x_l]$ the symmetric algebra. 
A hyperplane arrangement $\A = \{H_1, \dots, H_n\}$ is a finite collection of hyperplanes in $V$.

Freeness of an arrangement is a key notion which connects arrangement theory with algebraic geometry and combinatorics. 
The study of free arrangements was started by Saito \cite{saito} and a remarkable factorization theorem was proved by Terao \cite{terao1981generalized}. 
This theorem asserts that the characteristic polynomial of a free arrangement completely factors into linear polynomials over the integers. 
This imposes a necessary condition on the structure of the intersection lattice for an arrangement to be free. The Terao conjecture is the converse problem,
i.e. to understand if the structure of the intersection lattice characterize freeness of arrangements. A lot of work has done to solve this conjecture especially in the case
of characteristic $0$ (see for example \cite{yoshinaga2004characterization}, \cite{yoshinaga2005freeness}, \cite{abe2013free}, \cite{abe2016divisionally-im} and \cite{Gin-freearr}). 
However, also the case of finite characteristic has been analyzed (see for example \cite{terao1983free} and \cite{yoshinaga2007free}).
Despite this effort, this question is still open.

The purpose of this paper is to study the connections between freeness over a field of characteristic zero and over a finite field, and to describe in which cases
 the two situations are related and how. 

\section{Preliminares on hyperplane arrangements}\label{sec:arr}

In this section, we recall the terminology, the basic notations and some fundamental results related to hyperplane arrangements.

Let $K$ be a field. A finite set of affine hyperplanes $\A =\{H_1, \dots, H_n\}$ in $K^l$ 
is called a \textbf{hyperplane arrangement}. For each hyperplane $H_i$ we fix a defining equation $\alpha_i\in S= K[x_1,\dots, x_l]$ such that $H_i = \alpha_i^{-1}(0)$, 
and let $Q(\A)=\prod_{i=1}^n\alpha_i$. An arrangement $\A$ is called \textbf{central} if each $H_i$ contains the origin of $K^l$. 
In this case, the defining equation $\alpha_i\in S$ is linear homogeneous, and hence $Q(\A)$ is homogeneous of degree $n$. 

Let $L(\A)=\{\bigcap_{H\in\mathcal{B}}H \mid \mathcal{B}\subseteq\A\}$ be the \textbf{lattice of intersection} of $\A$, 
ordered by reverse inclusion, i.e. $X\le Y$ if and only if $Y\subseteq X$, for $X,Y\in L(\A)$.
Define a rank function on $L(\A)$ by $\rk(X)=\codim(X)$. 
$L(\A)$ plays a fundamental role in the study of hyperplane arrangements, in fact it determines the combinatorics of the arrangement.

Let $\mu\colon L(\A)\to\mathbb{Z}$ be the \textbf{M\"obius function} of $L(\A)$ defined by
$$\mu(X)=
\begin{cases}
      1 & \text{for } X=K^l,\\
      -\sum_{Y<X}\mu(Y) & \text{if } X>K^l.
\end{cases}$$
\begin{Definition}The \textbf{characteristic polynomial} of $\A$ is defined by $$\chi(\A,t) = \sum_{X\in L(\A)}\mu(X)t^{\dim(X)}.$$
\end{Definition}
%

The importance of the characteristic polynomial in combinatorics is justified by the following result 
from \cite{crapo1970foundations}, \cite{orlik1980combinatorics} and \cite{zaslavsky1975facing}.

\begin{Theorem} 
We have that
\begin{enumerate}
\item If $\A$ is an arrangement in $\mathbb{F}_p^l$, then $|\mathbb{F}^l_p\setminus \bigcup_{H\in\A}H|=\chi(\A, p)$.
\item If $\A$ is an arrangement in $\C^l$, then the topological $i$-th Betti number of the complement is $b_i(\C^l  \setminus \bigcup_{H\in\A}H)=b_i(\A)$.
\item If $\A$ is an arrangement in $\R^l$, then $|\chi(\A,-1)|$ is the number of chambers and $|\chi(\A, 1)|$ is the number of bounded chambers.
\end{enumerate}
\end{Theorem}

\section{Free hyperplane arrangements}
We recall the basic notions and properties of free hyperplane arrangements.

We denote by $\Der_{K^l} =\{\sum_{i=1}^l f_i\partial_{x_i}~|~f_i\in S\}$ the $S$-module of \textbf{polynomial vector fields} on $K^l$ (or $S$-derivations). 
Let $\delta =  \sum_{i=1}^l f_i\partial_{x_i}\in \Der_{K^l}$. Then $\delta$ is  said to be \textbf{homogeneous of polynomial degree} $d$ if $f_1, \dots, f_l$ are homogeneous polynomials of degree~$d$ in $S$. 
In this case, we write $\pdeg(\delta) = d$.

\begin{Definition} 
Let $\A$ be a central arrangement in $K^l$. Define the \textbf{module of vector fields logarithmic tangent} to $\A$ (or logarithmic vector fields) by
$$D(\A) = \{\delta\in \Der_{K^l}~|~ \delta(\alpha_i) \in (\alpha_i) S, \forall i\}.$$
\end{Definition}

The module $D(\A)$ is obviously a graded $S$-module and we have that $D(\A)= \{\delta\in \Der_{K^l}~|~ \delta(Q(\A)) \in (Q(\A)) S\}$. 
In particular, since the arrangement $\A$ is central, then the Euler vector field $\delta_E=\sum_{i=1}^lx_i\partial_{x_i}$ belongs to $D(\A)$, in fact
$\delta_E(Q(\A))=nQ(\A)$. 
In this case, if the characteristic of $K$ does not divide $n$, we can write $D(\A)\cong S{\cdot}\delta_E\oplus D_0(\A)$, where $$D_0(\A)=\{\delta\in \Der_{K^l}~|~ \delta(Q(\A))=0\}.$$

\begin{Definition} 
A central arrangement $\A$ in $K^l$ is said to be \textbf{free with exponents $(e_1,\dots,e_l)$} 
if and only if $D(\A)$ is a free $S$-module and there exists a basis $\delta_1,\dots,\delta_l \in D(\A)$ 
such that $\pdeg(\delta_i) = e_i$, or equivalently $D(\A)\cong\bigoplus_{i=1}^lS(-e_i)$.
\end{Definition}

In general the exponents of an arrangement depend on the characteristic of $K$. In fact, we have the following examples.
\begin{Example}[\cite{orlterao}, Example 4.35]\label{ex:specialexample} Consider the arrangement $\A$ in $K^3$ with defining equation
$Q(\A)=xyz(x-y)(x+z)(y+z)(x+y+z)$. Then $\A$ is free for any $K$, but its exponents depend on the characteristic of $K$.

If $\charact(K)\ne2$, then $\A$ is free with exponents $(1,3,3)$, in fact we can take as basis of $D(\A)$ the following vector fields $\delta_E,$ 
$\delta_2=x(x+z)(x+y+z)\partial_x+y(y+z)(x+y+z)\partial_y$ and $\delta_3=x(x+z)(2y+z)\partial_x+y(y+z)(2x+z)\partial_y$.

If $\charact(K)=2$, then $\A$ is free with exponents $(1,2,4)$, in fact we can take as basis of $D(\A)$ the following vector fields $\delta_E,$ 
$\delta_2=x^2\partial_x+y^2\partial_y+z^2\partial_z$ and $\delta_3=x^4\partial_x+y^4\partial_y+z^4\partial_z$.
\end{Example}

\begin{Example}\label{ex:specialexample2fields} Consider $\A$ the arrangement in $K^3$ as the cone of $\A^{[-2,2]}$ the Shi-Catalan arrangement of type B , see for example \cite{abe2011freeness}.
Then $\A$ is free for any $K$, but its exponents depend on the characteristic of $K$.

If $\charact(K)=5$, then $\A$ is free with exponents $[1,5,15]$. If $\charact(K)=7$, then $\A$ is free with exponents $[1,7,13]$.
If $\charact(K)\ne5,7$, then $\A$ is free with exponents $[1,9,11]$.
\end{Example}

\begin{Remark}\label{rem:reducttoD0}
Assume that the characteristic of $K$ does not divide $n$, then a central arrangement $\A$ is free if and only if $D_0(\A)$ is a free $S$-module.
\end{Remark}

In general, it is important to distinguish between the case that the characteristic of $K$ does not divide $n$ and the case that it does.
In the latter, there might not exist $\delta\in D(\A)$ such that $\delta(Q(\A))$ is a non-zero scalar multiple of $Q(\A)$.
\begin{Example} Consider the arrangement $\A$ in $\mathbb{F}_3^3$ with defining equation $Q(\A)=xyz(y+z)(x-y)(x+z)$. In this case $n=6$. Then
$\A$ is free and we can take $\{\delta_E,\delta_2,\delta_3\}$ as basis of $D(\A)$, where $\delta_2=(xy-y^2)\partial_y+(xz+z^2)\partial_z$ and 
$\delta_3=(xyz +xz^2 +yz^2 +z^3)\partial_z$. In this situation, $\delta_E(Q(\A))=\delta_2(Q(\A))=0$ and $\delta_3(Q(\A))=xQ(\A)$.
Hence, there does not exist $\delta\in D(\A)$ such that $\delta(Q(\A))$ is a non-zero scalar multiple of $Q(\A)$.
\end{Example}


Let $\delta_1,\dots,\delta_l \in D(\A)$. 
Then $\det(\delta_i(x_j))_{i,j}$ is divisible by $Q(\A)$.
One of the most famous characterization of freeness is due to Saito \cite{saito} and it uses  the determinant of the coefficient 
matrix of $\delta_1,\dots,\delta_l$ to check if the arrangement $\A$ is free or not. Notice that the original statement is for
characteristic $0$, but in \cite{terao1983free} Terao showed that this statement holds true for any characteristic.

\begin{Theorem}[Saito's criterion]\label{theo:saitocrit}
Let $\A$ be a central arrangement in $K^l$ and $\delta_1, \dots, \delta_l \in D(\A)$. Then the following facts are equivalent
\begin{enumerate}
\item $D(\A)$ is free with basis $\delta_1, \dots, \delta_l$, i. e. $D(\A) = S\cdot\delta_1\oplus \cdots \oplus S \cdot\delta_l$.
\item $\det(\delta_i(x_j))_{i,j}=c Q(\A)$, where $c\in K\setminus\{0\}$.
\item $\delta_1, \dots, \delta_l$ are linearly independent over $S$ and $\sum_{i=1}^l\pdeg(\delta_i)=n$.
\end{enumerate}
\end{Theorem}


Given an arrangement $\A$ in $K^l$, the \textbf{Jacobian ideal} of $\A$
is the ideal of $S$ generated by $Q(\A)$ and all its partial derivatives, and it is denoted by $J(\A)$.
This ideal has a central role in the study of free arrangements.
In fact, we can also characterize freeness by looking at the Jacobian ideal of $\A$.
This characterization by Terao \cite{terao1980arrangementsI} will be used in Section \ref{sec:fromPtozerochar}.
Notice that Terao described this result for characteristic $0$, but the same proof also works for 
the case that the characteristic does not divide the cardinality of $\A$.

\begin{Theorem}[Terao's criterion]\label{theo:freCMcod2} 
A central arrangement $\A$ in $K^l$ is free if and only if $S/J(\A)$ is $0$ or $(l-2)$-dimensional Cohen-Macaulay.
\end{Theorem}
\begin{proof}
By Remark \ref{rem:reducttoD0}, we need to prove that $D_0(\A)$ is a free $S$-module if and only if $S/J(\A)$ is $0$ or Cohen-Macaulay. 

We have an exact sequence $$\xymatrix{0\ar[r]&D_0(\A)\ar[r]^{~~~\alpha} &S^l \ar[r]^\beta &S \ar[r]^{\gamma~~~~} &S/J(\A)\ar[r] &0},$$ where $$\alpha(\sum_{i=1}^l f_i \partial / \partial x_i)=(f_1, \dots, f_l)^t ~\text{ for }~ \sum_{i=1}^l f_i \partial / \partial x_i \in D_0(\A),$$ $$\beta((g_1, \dots, g_l)^t)= \sum_{i=1}^l g_i \partial Q(\A)/ \partial x_i ~\text{ for }~ (g_1, \dots, g_l)^t\in S^l$$ and $\gamma$ is the natural projection.
Thus $D_0(\A)$ is free if and only if the homological dimension of $S / J(\A)$ is less than three.

Recall the Auslander-Buchsbaum equality $$\depth(S / J(\A))+\hdim(S / J(\A))=\dim(S)=l,$$ if $S / J(\A)\ne 0$. On the other hand, we have $\dim(S / J(\A))\le l-2$, then $$ \hdim(S / J(\A))\le 2$$ $$\Leftrightarrow \depth(S / J(\A))\ge l-2 \ge \dim(S / J(\A))$$ $$\Leftrightarrow \depth(S / J(\A))= l-2 = \dim(S / J(\A))$$ $$\Leftrightarrow S / J(\A) \text{ is Cohen-Macaulay of dimension } l-2,$$ if $S / J(\A)\ne 0$. When $S / J(\A)=0$, $D_0(\A)$ is obviously a free $S$-module. 
\end{proof}


%

Freeness has several consequences. For example we recall the following.

\begin{Theorem}\emph{(\cite{terao1981generalized})} 
Suppose that $\A$ is a free arrangement in $K^l$ with exponents $(e_1, \dots , e_l)$. 
Then $$\chi(\A,t)=\prod_{i=1}^l(t-e_i).$$
\end{Theorem}

The previous theorem imposes a necessary condition on the structure of  $L(\A)$ for the arrangement $\A$ to be free. 
The Terao conjecture is the problem to ask the converse, i.e if the structure of $L(\A)$ characterize freeness of $\A$. 
This conjecture is still unsolved.  In \cite{yoshinaga2007free}, Yoshinaga gave an affirmative 
result when $K=\mathbb{F}_p$  and $l = 3$. Specifically, he proved the following result.
\begin{Theorem}\label{theo:yoshcharpfreechar} Let $\A$ be a central arrangement in $\mathbb{F}_p^3$. Then the following facts hold true.
\begin{enumerate}
\item When $|\A|\ge2p$, $\A$ is free if and only if $\chi(\A,p)=0$.
\item When $|\A|=2p-1$, $\A$ is free if and only if either $\chi(\A,p)=0$ or $\chi(\A,t)=(t-1)(t-p+1)^2$.
\item When $|\A|=2p-2$, $\A$ is free if and only if either $\chi(\A,p)=0$ or $\chi(\A,t)=(t-1)(t-p+1)(t-p+2)$.
\end{enumerate}
\end{Theorem}
More in general, in \cite{yoshinaga2007free}, Yoshinaga also proved the following.
\begin{Theorem}\label{theo:yoshcharpfreel} Let $\A$ be a central arrangement in $\mathbb{F}_p^l$. If $\chi(\A,p^{l-2})=0$, then
$\A$ is free with exponents $(1,p,\dots,p^{l-2},|\A|-1-p-\cdots-p^{l-2})$.
\end{Theorem}

\section{From characteristic $0$ to characteristic $p$}\label{sec:freefromchar0tocharp}

From now on we will assume that $\A=\{H_1,\dots, H_n\}$ is a central arrangement in $\mathbb{Q}^l$. After getting rid of the denominators, we can suppose that $\alpha_i \in\mathbb{Z}[x_1,\dots, x_l]$ for all $i=1,\dots, n$, and hence that $Q(\A)=\prod_{i=1}^n\alpha_i\in\mathbb{Z}[x_1,\dots, x_l]$. Moreover, we can also assume that there exists no prime number $p$ that divides any $\alpha_i$.

Let $p$ be a prime number. Consider the image of $Q(\A)$ under the canonical homomorphism $\pi_p\colon \mathbb{Z}[x_1,\dots, x_l]\to \mathbb{F}_p[x_1,\dots, x_l]$. 
\begin{Definition} Let $\A$ be a central arrangement in $\mathbb{Q}^l$. We will call a prime number $p$ \textbf{good for $\A$} if $\pi_p(Q(\A))$ is reduced.
\end{Definition}

\begin{Lemma}\label{lem:finitenotgoodprimes} There is a finite number of primes $p$ that are not good for $\A$.
\end{Lemma}
\begin{proof} We have that $\pi_p(Q(\A))$ is not reduced if and only if $\pi_p(\alpha_i)=\pi_p(\alpha_j)$ for some $i\ne j$ and this can happens only for a finite number of primes.
\end{proof}

Let now $p$ be a good prime for $\A$, and consider $\A_p$ the arrangement in $\mathbb{F}_p^l$ defined by $\pi_p(Q(\A))$. Hence, by construction, $Q(\A_p)=\pi_p(Q(\A))\ne0$ and it is reduced.

\begin{Theorem}\label{theo:fromchar0tocharP} If $\A$ is free in $\mathbb{Q}^l$ with exponents $(e_1,\dots, e_l)$, then $\A_p$ is free in $\mathbb{F}_p^l$ with exponents $(e_1,\dots, e_l)$, 
for all good primes except possibly a finite number of them.
\end{Theorem}
\begin{proof} Let $\Delta=\{\delta_1,\dots, \delta_l\}$ be a basis of $D(\A)$ with $\pdeg(\delta_i)=e_i$ for all $i=1,\dots, l$. Since $\A$ has a defining equation with only integer coefficients, 
we can assume that every polynomial that appear in each $\delta\in\Delta$ is in $\mathbb{Z}[x_1,\dots, x_l]$. Hence we can consider $\bar{\delta}_1,\dots, \bar{\delta}_l\in \Der_{\mathbb{F}_p^l}$ 
 the image of $\delta_1,\dots, \delta_l$. 
We can assume that $p\nmid\delta$ for all $\delta\in\Delta$, and hence $\bar{\delta}\ne0$ for all $\delta\in\Delta$. This implies that $\pdeg(\bar{\delta}_i)=\pdeg(\delta_i)=e_i$ for all $i=1,\dots, l$.

By the definition of $D(\A)$, for each $\delta\in \Delta$ there exists $h\in\mathbb{Z}[x_1,\dots, x_l]$ such that $\delta(Q(\A))=hQ(\A)$. If we apply $\pi_p$ to this expression we obtain that $\bar{\delta}(Q(\A_p))=\pi_p(\delta(Q(\A)))=\pi_p(hQ(\A))=\pi_p(h)Q(A_p)$, and hence $\bar{\delta}\in D(\A_p)$.

By Theorem \ref{theo:saitocrit}, since in each $\delta\in\Delta$  every polynomial that appear have only integer coefficients, there exists $c\in\mathbb{Z}\setminus\{0\}$ such that $\det(\delta_i(x_j))_{i,j}=c Q(\A)$. 
If we apply $\pi_p$ to the previous equality we obtain that $\det(\bar{\delta}_i(x_j))_{i,j}=\pi_p(\det(\delta_i(x_j))_{i,j})=\pi_p(c Q(\A))=\pi_p(c)Q(\A_p)$. Hence if $p$ does not divide $c$, 
we have that $\pi_p(c)\in\mathbb{F}_p\setminus\{0\}$ and hence again by Theorem \ref{theo:saitocrit}, we have that $\bar{\delta}_1,\dots, \bar{\delta}_l$ are a basis of $D(\A_p)$. 
This proves that $\A_p$ is free with exponents $(e_1,\dots, e_l)$.
\end{proof}

By Lemma \ref{lem:finitenotgoodprimes}, the number of non good primes is finite. Hence we have the following.
\begin{Corollary}\label{corol:fromchar0tocharP} Let $\A$ be a central arrangement in $\mathbb{Q}^l$ and $p$ a large prime number. If $\A$ is free in $\mathbb{Q}^l$ with exponents $(e_1,\dots, e_l)$, then $\A_p$ is free in $\mathbb{F}_p^l$ with exponents $(e_1,\dots, e_l)$.
\end{Corollary}

From Example \ref{ex:specialexample}, we have the following.
\begin{Example} Consider the arrangement $\A$ of Example \ref{ex:specialexample} with $K=\mathbb{Q}$. Then the determinant of the coefficient matrix of the basis of $D(\A)$ is equal to
$2Q(\A)$. A direct computation shows that if we take another basis of $D(\A)$ with only integer coefficients, then the determinant of the coefficient matrix is equal to $cQ(\A)$ with $c\in 2\mathbb{Z}\setminus\{0\}$.
This is why over $\mathbb{F}_2$ the exponents of $\A$ change.
\end{Example}

Similarly to the previous example, we have the following

\begin{Example} Consider $\A$ the cone of $\A^{[-2,2]}$ the Shi-Catalan arrangement of type B  as in Example \ref{ex:specialexample2fields} with $K=\mathbb{Q}$. 
Then the determinant of the coefficient matrix of any basis of $D(\A)$ with only integer coefficients is equal to
$c Q(\A)$, with $c\in 35\mathbb{Z}\setminus\{0\}$. This is why over $\mathbb{F}_5$ and $\mathbb{F}_7$ the exponents of $\A$ change.
\end{Example}
 
\begin{Example}[\cite{ziegler1990matroid}] Let $\A_{all}$ be the arrangement of all hyperplanes of $\mathbb{F}_3^3$ and $L$ a line in $\mathbb{F}_3^3$. Consider the arrangement
$\A=\{H\in\A_{all}~|~H\nsupseteq L\}$ in $\mathbb{F}_3^3$. Now $\A$ is not free. However, if we consider $K$ a field such that $\charact(K)\ne3$ and $\A'$ an arrangement in $K^3$ satisfying $L(\A)\cong L(\A')$,
then $\A'$ is free with exponents $(1,4,4)$. Notice that the corresponding matroid is not representable over $\R$ or $\mathbb{Q}$.
\end{Example}
 

 The following is an example of a free arrangements in $\mathbb{Q}^l$ that is not free in $\mathbb{F}_p^l$, for some good prime $p$.
 \begin{Example} Consider $\A$ the arrangement in $\mathbb{Q}^4$ as the cone of $\A^{[-2,2]}$ the Shi-Catalan arrangement of type B. As described in \cite{abe2011freeness},
 $\A$ is free with exponents $(1,13,15,17)$. Moreover, the determinant of the coefficient matrix of a basis of $D(\A)$ with only integer coefficients is equal to
$c Q(\A)$, with $c\in 56595\mathbb{Z}\setminus\{0\}$. Notice that $56595=3\cdot5\cdot7^3\cdot11$. Now, $3$ is not a good prime but $5,7$ and $11$ are.
A direct computation shows that the arrangement $\A_5$ over $\mathbb{F}_5$ is free with exponents $(1,5,15,25)$. However, both $\A_7$ over $\mathbb{F}_7$
and $\A_{11}$ over $\mathbb{F}_{11}$ are not free.
 \end{Example}

\section{Homogeneous ideals in $\mathbb{Z}[x_1,\dots, x_l]$}
In this section, we study ideals in $\mathbb{Z}[x_1,\dots, x_l]$, their Betti numbers and their zero divisors. 
This will play an important role in Section \ref{sec:fromPtozerochar}.

\begin{Definition} Given a sequence of numbers $\{c_{i,j}\}$, we obtain a new sequence by a \textbf{cancellation} as follows. 
Fix a $j$, and choose $i$ and $i'$ so that one of the numbers is odd and the other is even. Then replace $c_{i,j}$ by $c_{i,j}-1$, 
and replace $c_{i',j}$ by $c_{i',j}-1$. We have a \textbf{consecutive cancellation} when $i' = i+1$.
\end{Definition}

The following result is a generalization of Proposition 5.2 from \cite{dalili2014dependence} and it will play an important role in Theorem \ref{theo:fromPtozero}.
\begin{Proposition}\label{prop:betticonsecutivcancel} Let $I$ be a homogeneous ideal of $R= \mathbb{Z}[x_1,\dots ,x_l]$ and $p$ a prime number that is a non-zero divisor in $R/I$. 
Then, the graded Betti numbers of $(R/I) \otimes_\mathbb{Z} \mathbb{Q}$ can be obtained from the one of $(R/I) \otimes_\mathbb{Z} \mathbb{F}_p$ by a sequence of consecutive cancellations.
\end{Proposition}
\begin{proof} Since $p$ is a non-zero divisor in $R/I$, then $(R/I) \otimes_\mathbb{Z} \mathbb{Z}_{(p)}$ is flat over $\mathbb{Z}_{(p)}$. 
Let $\mathbb{F}_\bullet$ be a minimal graded free $R\otimes_\mathbb{Z}\mathbb{Z}_{(p)}$-resolution of $(R/I) \otimes_\mathbb{Z} \mathbb{Z}_{(p)}$. 
Note that $\mathbb{F}_\bullet \otimes_\mathbb{Z} \mathbb{Z}_{(p)}$ is a free resolution of
$(R/I)\otimes_\mathbb{Z} \mathbb{Z}_{(p)}$ over $R \otimes_\mathbb{Z} \mathbb{Z}_{(p)} = \mathbb{Z}_{(p)}[x_1, \ldots, x_l]$, and 
that the graded Betti numbers of $\mathbb{F}_\bullet$ and $\mathbb{F}_\bullet \otimes_\mathbb{Z} \mathbb{Z}_{(p)}$ are the same.
The flatness of $(R/I) \otimes_\mathbb{Z} \mathbb{Z}_{(p)}$ over $\mathbb{Z}_{(p)}$ ensures that the complex $\mathbb{F}_\bullet \otimes_\mathbb{Z} \mathbb{F}_p =
\mathbb{F}_\bullet \otimes_\mathbb{Z} \mathbb{Z}_{(p)} \otimes_{\mathbb{Z}_{(p)}} \mathbb{F}_p$
is a free resolution of $(R/I) \otimes_\mathbb{Z} \mathbb{F}_p = (R/I) \otimes_\mathbb{Z} \mathbb{Z}_{(p)} \otimes_{\mathbb{Z}_{(p)}} \mathbb{F}_p$
over $R\otimes_\mathbb{Z} \mathbb{F}_p = \mathbb{F}_p[x_1, \dots, x_l]$.
Now matrices giving the maps in this complex are in $(x_1, \ldots, x_l)$, i.e., it is a minimal resolution. This shows that the graded Betti number of $\mathbb{F}_\bullet$ 
coincides with the one of  $(R/I) \otimes_\mathbb{Z} \mathbb{F}_p$.

Now, $\mathbb{F}_\bullet \otimes_{\mathbb{Z}_{(p)}} \mathbb{Q}$ is a graded free $(R\otimes_\mathbb{Z}\mathbb{Q})$-resolution of $(R/I)\otimes_\mathbb{Z}\mathbb{Q}$. 
Therefore we can write $\mathbb{F}_\bullet \otimes_{\mathbb{Z}_{(p)}} \mathbb{Q}=\mathbb{G}_\bullet\oplus\mathbb{G}'_\bullet$,
 where $\mathbb{G}_\bullet$ is a minimal graded free $(R\otimes_\mathbb{Z}\mathbb{Q})$-resolution of $(R/I)\otimes_\mathbb{Z}\mathbb{Q}$ and $\mathbb{G}'_\bullet$ is a 
 graded trivial complex of free $(R\otimes_\mathbb{Z}\mathbb{Q})$-modules \cite[Theorem 20.2]{eisenbud}. Therefore, the graded Betti numbers of 
 $(R/I) \otimes_\mathbb{Z} \mathbb{Q}$ can be obtained from the one of $(R/I) \otimes_\mathbb{Z} \mathbb{F}_p$ by a sequence of consecutive cancellations.
\end{proof}

From the notion of consecutive cancellations, we directly obtain the relation between homological dimensions.

\begin{Corollary}\label{corol:bettiinequalQtoP} Let $I$ be a homogeneous ideal of $R= \mathbb{Z}[x_1,\dots ,x_l]$ and $p$ a prime number that is a non-zero divisor in $R/I$. 
Then, the Betti numbers of $(R/I) \otimes_\mathbb{Z} \mathbb{Q}$ are smaller or equal to the one of $(R/I) \otimes_\mathbb{Z} \mathbb{F}_p$, and hence
$\hdim((R/I) \otimes_\mathbb{Z} \mathbb{Q})\le\hdim((R/I) \otimes_\mathbb{Z} \mathbb{F}_p)$.
\end{Corollary}

In general, given $I$ an ideal of $R= \mathbb{Z}[x_1,\dots ,x_l]$, the number of zero divisor in $R/I$ is infinite.
However, if we restrict our attention to zero divisors that are prime numbers, we have the following.
\begin{Proposition}\label{prop:finitezerodivprimes} Let $I$ be an ideal of $R=\mathbb{Z}[x_1,\dots ,x_l]$. Then the number of distinct prime numbers
that are zero divisors in $R/I$ is finite.
\end{Proposition}
\begin{proof}  By Theorem 14.4 of \cite{eisenbud}, there exists $a\in\mathbb{Z}\setminus\{0\}$ such that $(R/I)[a^{-1}]$ is a free $\mathbb{Z}[a^{-1}]$-module. This implies that 
the set of distinct prime numbers that are zero divisors in $R/I$ is included in the set of distinct prime numbers that divide $a$, that is finite by the unique factorization theorem.
\end{proof}

Also if the number of prime numbers that are zero divisors in $R/I$ is finite, it might be difficult to compute them.
The rest of this section is dedicated to describe a method to compute them.

Consider $I$ an ideal of $R = \mathbb{Z}[x_1,\dots ,x_l]$ and $\sigma$ a term ordering on $R$. As for the case with polynomial rings over fields, we can study the  theory of Gr\"obner basis
for ideals of $R$. We refer to Chapter 4.5 of \cite{adams1994introduction} for details.
\begin{Definition} Let $I$ be an ideal of $R=\mathbb{Z}[x_1,\dots ,x_l]$, $\sigma$ a term ordering on $R$ and $G=\{g_1,\dots, g_t\}$ a set of non-zero polynomials in $I$.
Then we say that $G$ is a \textbf{minimal strong $\sigma$-Gr\"obner basis for $I$} if the following conditions hold true
\begin{enumerate}
\item $G$ forms a set of generators of $I$;
\item for each $f\in I$, there exists $i\in\{1,\dots,t\}$ such that $\LM_\sigma(g_i)$ divides $\LM_\sigma(f)$;
\item if $i\ne j$, then $\LM_\sigma(g_i)$ does not divide $\LM_\sigma(g_j)$.
\end{enumerate}
\end{Definition}

\begin{Remark}[c.f. \cite{adams1994introduction}, Lemma 4.5.8]\label{rem:minstronggbiggb} A minimal strong $\sigma$-Gr\"obner basis of an ideal $I$ of $\mathbb{Z}[x_1,\dots ,x_l]$ is also a $\sigma$-Gr\"obner basis of $I$. 
\end{Remark}

\begin{Proposition}[\cite{adams1994introduction}, Exercise 4.5.9] Let $I$ be a non-zero ideal of $R=\mathbb{Z}[x_1,\dots ,x_l]$ and $\sigma$ a term ordering on $R$. 
Then there always exists a minimal strong $\sigma$-Gr\"obner basis of $I$
\end{Proposition}
Notice that in general, a minimal strong $\sigma$-Gr\"obner basis is not unique. Consider the following example.
\begin{Example} Consider the ideal $I=(x^2-y,3y)$ in $\mathbb{Z}[x,y]$ and $\sigma=\degrevlex$. Then both 
$\{x^2-y,3y\}$ and $\{x^2+y,3y\}$ are minimal strong $\sigma$-Gr\"obner basis for $I$.
\end{Example}
\begin{Lemma}\label{lem:sameLMandLC} Let $I$ be an ideal of $R=\mathbb{Z}[x_1,\dots,x_l]$, and $\sigma$ a term ordering on $R$. Let $G_1$ and $G_2$ be two minimal strong 
$\sigma$-Gr\"obner bases of $I$. Then $\{\LM_\sigma(g)~|~g\in G_1\} = \{\LM_\sigma(g)~|~g\in G_2\}$. Consequently we have $|G_1| = |G_2|$ 
and $\{\LC_\sigma(g)~|~g\in G_1\} = \{\LC_\sigma(g)~|~g\in G_2\}$.
\end{Lemma}
\begin{proof}
Assume that there exists $t\in\LM_\sigma(G_1)\setminus\LM_\sigma(G_2)$. Now, $t\in \LM_\sigma(I)$ and hence there exists $f\in I$ such that $t=\LM_\sigma(f)$.
Since $G_2$ is a minimal strong $\sigma$-Gr\"obner basis for $I$, there exists $g\in G_2$ such that $\LM_\sigma(g)=t'$ divides $\LM_\sigma(f)=t$. With a similar argument, we see that $t'$, and hence $t$,
is a multiple of an element $t''$ in $\LM_\sigma(G_1)$. The minimality of $G_1$ implies that $t''=t$. Hence, $t=t'\in \LM_\sigma(G_2)$, but this is a contradiction.
This shows that $\LM_\sigma(G_1)\subseteq\LM_\sigma(G_2)$. 

With the same argument we can show that $\LM_\sigma(G_2)\subseteq\LM_\sigma(G_2)$, and hence that $\LM_\sigma(G_1)=\LM_\sigma(G_2)$.
\end{proof}

\begin{Remark} The previous lemma implies that $\{\LM_\sigma(g)~|~g\in G\}$ generates the monomial ideal $\LM_\sigma(I)$.
\end{Remark}

By Lemma \ref{lem:sameLMandLC}, we can introduce the following Definition. See \cite{pauer1992lucky}, for more details.

\begin{Definition} Let $I$ be an ideal of $R=\mathbb{Z}[x_1,\dots,x_l]$, and $\sigma$ be a term ordering on $R$. 
If a prime number $p$ does not divide the leading coefficient of any polynomial 
in a minimal strong $\sigma$-Gr\"obner basis for $I$, then we will call $p$ \textbf{$\sigma$-lucky for $I$}.
\end{Definition}
\begin{Remark}\label{rem:finiteluckyprime} Given $I$ an ideal of $R=\mathbb{Z}[x_1,\dots,x_l]$ and $\sigma$ be a term ordering on $R$, since a minimal strong $\sigma$-Gr\"obner basis is finite,
 then the number of primes that are not $\sigma$-lucky for $I$ is finite.
\end{Remark}
Notice that the list of $\sigma$-lucky primes for an ideal depends on the choice of term ordering.
\begin{Example}\label{ex:luckyprimeschange} Consider the ideal $I=(3x-y)$ in $\mathbb{Z}[x,y]$. Consider on $\mathbb{Z}[x,y]$ a term ordering $\sigma$
such that $x>_\sigma y$. Then $\{3x-y\}$ is a minimal strong $\sigma$-Gr\"obner basis of $I$, and hence every prime $p\ne 3$ is $\sigma$-lucky for $I$. 
However, if we consider on $\mathbb{Z}[x,y]$ a term ordering $\tau$
such that $y>_\tau x$, then $\{y-3x\}$ is a minimal strong $\tau$-Gr\"obner basis of $I$. Hence every prime $p$ is $\tau$-lucky for $I$, also $p=3$.
\end{Example}

Notice that in general, we cannot determine the $\sigma$-lucky primes from the coefficients that appear in a set of generators.
\begin{Example} Consider the ideal $I=(2x+3y,x-y)$ in $\mathbb{Z}[x,y]$ with $\sigma=\degrevlex$. In this case,
$\{5y,x-y\}$ is a minimal strong $\sigma$-Gr\"obner basis of $I$, and hence $p=5$ is the only non $\sigma$-lucky prime of $I$.
\end{Example}

We are now ready to describe the relation between lucky primes and primes that are non-zero divisors.
\begin{Proposition}\label{prop:luckyisnzd} Let $I$ be an ideal of $R=\mathbb{Z}[x_1,\dots,x_l]$ and $\sigma$ a term ordering on $R$. If $p$ is a 
$\sigma$-lucky prime for $I$, then $p$ is a non-zero divisor in $R/I$.
\end{Proposition}
\begin{proof} Let $G=\{g_1,\dots,g_t\}$ be a minimal strong $\sigma$-Gr\"obner basis of $I$ and $p$ a $\sigma$-lucky prime for $I$. 
Suppose that $p$ is a zero divisor in $R/I$. This implies that there exists $f\in R\setminus I$ such that $pf\in I$. By definition of minimal strong $\sigma$-Gr\"obner basis,
there exists $i\in\{1,\dots,t\}$ such that $\LM_\sigma(g_i)$ divides $\LM_\sigma(pf)$. We can now consider $h_1=pf-\frac{\LM_\sigma(pf)}{\LM_\sigma(g_i)}g_i\in I$. Since $p$ $\sigma$-lucky,
$p| \frac{\LM_\sigma(pf_1)}{\LM_\sigma(g_j)}$, and hence $h_1=pf_1$, for some $f_1\in R$. 
We can now repeat the process with $pf_1\in I$ and obtain $pf_2=pf_1-\frac{\LM_\sigma(pf_1)}{\LM_\sigma(g_j)}g_j\in I$, for some $j$. 
Since $G$ is a $\sigma$-Gr\"obner basis this process will end after a finite number of steps, i.e. we will obtain that there exists a $k\ge1$ such that $pf_k\in I$ with $pf_k\ne0$ but $pf_{k+1}=0$.
Retrieving each step, we can write $pf=\sum_{i=1}^tpa_ig_i$, for some $a_i\in R$. This implies that $f=\sum_{i=1}^ta_ig_i\in I$, but this is impossible by assumption.
\end{proof}

Notice that in general the set of distinct prime numbers that are non-zero divisors in $R/I$ contains strictly the set of $\sigma$-lucky primes for $I$.
\begin{Example} Consider the ideal of Example \ref{ex:luckyprimeschange} with the term ordering~$\sigma$. Then $p=3$ is a non-zero divisor in $R/I$, but
it is not a $\sigma$-lucky prime.
\end{Example}

\section{From characteristic $p$ to characteristic $0$}\label{sec:fromPtozerochar}
As in Section \ref{sec:freefromchar0tocharp}, we will assume that $\A=\{H_1,\dots, H_n\}$ is a central arrangement in $\mathbb{Q}^l$ and hence that $\alpha_i \in\mathbb{Z}[x_1,\dots, x_l]$ for all $i=1,\dots, n$.
Moreover, we can also assume that there exists no prime number $p$ that divides any $\alpha_i$.

In the rest of the paper, we will denote by $J(\A)_\mathbb{Z}$ the ideal of $\mathbb{Z}[x_1,\dots ,x_l]$ generated by $Q(\A)$ and its partial derivatives.

\begin{Theorem}\label{theo:fromPtozero} Let $\A=\{H_1,\dots, H_n\}$ be a central arrangement in $\mathbb{Q}^l$. Let $p$
be a good prime number for $\A$ that does not divide $n$ and that is a non-zero divisor in $\mathbb{Z}[x_1,\dots ,x_l]/J(\A)_\mathbb{Z}$. If $\A_p$ is free in $\mathbb{F}_p^l$ with exponents $(e_1,\dots, e_l)$, 
then $\A$ is free in $\mathbb{Q}^l$ with exponents $(e_1,\dots, e_l)$.
\end{Theorem}
\begin{proof} Denote $S=\mathbb{Q}[x_1,\dots, x_l]$, $S_p=\mathbb{F}_p[x_1,\dots, x_l]$ and $R=\mathbb{Z}[x_1,\dots ,x_l]$. 
Notice that  $S_p/J(\A_p)\cong (R/J(\A)_\mathbb{Z})\otimes_{\mathbb{Z}}\mathbb{F}_p$ and
$S/J(\A)\cong (R/J(\A)_\mathbb{Z})\otimes_{\mathbb{Z}}\mathbb{Q}$.

Clearly $S_p/J(\A_p)=0$ if and only if $S_/J(\A)=0$

Assume now that $S_p/J(\A_p)\ne0$. By Theorem \ref{theo:freCMcod2}, $S_p/J(\A_p)$ is Cohen-Macaulay of dimension
$l-2$.

By Auslander-Buchsbaum equality, we can write
$$\depth(S_p / J(\A_p))+\hdim(S_p / J(\A_p))=l=\depth(S/ J(\A))+\hdim(S/ J(\A)).$$
By Corollary \ref{corol:bettiinequalQtoP}, $\hdim(S_p/ J(\A_p))\ge\hdim(S/ J(\A))$ and hence, by the previous equality, $\depth(S_p/ J(\A_p))\le\depth(S/ J(\A))$.
Now $S_p/J(\A_p)$ is Cohen-Macaulay, hence $\depth(S_p / J(\A_p))=\dim(S_p / J(\A_p))=l-2=\dim(S/ J(\A))$. This implies that $\depth(S/ J(\A))\ge l-2=\dim(S/ J(\A))$.
On the other hand, $\depth(S/ J(\A))\le\dim(S/ J(\A))$ and hence we have that $\depth(S/ J(\A))=\dim(S/ J(\A))$. This implies that $S/J(\A)$ is Cohen-Macaulay of dimension
$l-2$, and so $\A$ is free in $\mathbb{Q}^l$.

This argument also shows that $2=\hdim(S_p/ J(\A_p))=\hdim(S/ J(\A))$, and hence, $S/ J(\A)$ have the same Betti numbers of $S_p / J(\A_p)$. Since the exponents of $\A$ can be
computed uniquely from the Betti numbers of $S/ J(\A)$, this implies that $\A$ has exponents $(e_1,\dots, e_l)$.
\end{proof}


In Theorem \ref{theo:fromPtozero}, the assumption that the prime $p$ is a non-zero divisor in $\mathbb{Z}[x_1,\dots ,x_l]/J(\A)_\mathbb{Z}$ is fundamental. In fact we have the following.
\begin{Example} Consider the arrangement $\A\subseteq\mathbb{Q}^3$ with defining equation $Q(\A)=z(x+2y-4z)(y+4z)(x+3y-6z)$. $\A$ is non free and both $2$ and $3$ are zero divisors in $\mathbb{Z}[x_1,\dots ,x_l]/J(\A)_\mathbb{Z}$.
In fact, we have that $3(y^2z^2+2yz^3-8z^4)\in J(\A)_\mathbb{Z}$ but $y^2z^2+2yz^3-8z^4\notin J(\A)_\mathbb{Z}$, and similarly 
$2(xyz^2+4y^2z^2+4xz^3+8yz^3-32z^4)\in J(\A)_\mathbb{Z}$ but $xyz^2+4y^2z^2+4xz^3+8yz^3-32z^4\notin J(\A)_\mathbb{Z}$. However, both $A_2$ and $A_3$ are free with
exponents $(1,1,2)$.
\end{Example}

By Proposition \ref{prop:finitezerodivprimes}, the number of prime numbers that are zero divisors in $\mathbb{Z}[x_1,\dots ,x_l]/J(\A)_\mathbb{Z}$ is finite. Hence, putting together Corollary \ref{corol:fromchar0tocharP} 
and Theorem \ref{theo:fromPtozero}, we have the following.
\begin{Corollary}\label{corol:equivfreenesschar0p} Let $\A$ be a central arrangement in $\mathbb{Q}^l$ and $p$ a large prime number. $\A_p$ is free in $\mathbb{F}_p^l$ with exponents $(e_1,\dots, e_l)$ if and only if
$\A$ is free in $\mathbb{Q}^l$ with exponents $(e_1,\dots, e_l)$.
\end{Corollary}

\section{Applications}
As in Sections \ref{sec:freefromchar0tocharp} and \ref{sec:fromPtozerochar}, we will assume that $\A=\{H_1,\dots, H_n\}$ is a central arrangement in $\mathbb{Q}^l$ and hence that $\alpha_i \in\mathbb{Z}[x_1,\dots, x_l]$ for all $i=1,\dots, n$.
Moreover, we can also assume that there exists no prime number $p$ that divides any $\alpha_i$.

From Theorems \ref{theo:yoshcharpfreel} and \ref{theo:fromPtozero}, we have the following.
\begin{Theorem} Let $\A=\{H_1,\dots, H_n\}$ be a central arrangement in $\mathbb{Q}^l$. Let $p$
be a good prime number for $\A$ that does not divide $n$ and that is a non-zero divisor in $\mathbb{Z}[x_1,\dots ,x_l]/J(\A)_\mathbb{Z}$. If $\chi(\A_p,p^{l-2})=0$, then
$\A$ is free with exponents $(1,p,\dots,p^{l-2},|\A|-1-p-\cdots-p^{l-2})$.
\end{Theorem}

As noted in \cite{ardila2007computing}, if $p$ is a large prime number, then $\A$ and $\A_p$ have isomorphic intersection lattices, and hence have the same characteristic polynomial. 
This fact, together with the finiteness of good prime numbers that are zero divisors in $\mathbb{Z}[x_1,\dots ,x_l]/J(\A)_\mathbb{Z}$, gives us the following.

\begin{Corollary} Let $\A$ be a central arrangement in $\mathbb{Q}^l$ and $p$
a large prime number. If $\chi(\A,p^{l-2})=0$, then $\A$ is free with exponents $(1,p,\dots,p^{l-2},|\A|-1-p-\cdots-p^{l-2})$.
\end{Corollary}


\paragraph{\textbf{Acknowledgements}} The authors would like to thank G. Caviglia, M. Kummini, N. Nakashima and M. Yoshinaga for many helpful discussions.

\bibliography{bibliothesis}{}
\bibliographystyle{plain}

\end{document}